\documentclass[11pt]{article}
\usepackage{amsmath,amssymb,amsthm,comment}
\usepackage{amsfonts}
\usepackage{color}
\usepackage{epsfig}
\usepackage{mathtools}
\newtheorem{thm}{Theorem}[section]
\newtheorem{lemma}[thm]{Lemma}

\newtheorem{remark}[thm]{Remark}
\newtheorem{prop}[thm]{Proposition}

\makeatletter
    
    \@addtoreset{equation}{section}
  \makeatother
\newcommand{\R}{\mathbb{R}}
\newcommand{\N}{\mathbb{N}}

\newcommand{\dt}{\partial_t}
\newcommand{\ep}{\varepsilon}
\newcommand{\dl}{\delta}
\newcommand{\lap}{\Delta}

\begin{document}
\title{Asymptotic Blow-up Behavior for the Semilinear Heat Equation \\with Super-exponential Nonlinearities}
\author{
Ryoto Ichiya\footnote{ichiya.ryoto.t2@dc.tohoku.ac.jp}\\
{\small Mathematical Institute, Tohoku University,}\\
{\small Sendai 980-8578, Japan}
}
\date{\small\today}
\maketitle
\begin{abstract}
 We consider the semilinear heat equation $u_t - \lap u = f(u)$ in $\Omega = B_R(0) \subset \R^n$ with super-exponential nonlinearities $f(u) = e^{u^p}u^q$ ($p>1$, $q \in \{0\}\cup [1,\infty)$), 
 nonnegative bounded radially symmetric initial data and 0-Dirichlet boundary condition.
 In this paper, we show the asymptotic blow-up behavior for nonnegative, radial type I blow-up solution. 
 More precisely, we prove that if $n \leq 2$, then such blow-up solution satisfies
 \begin{equation*}
  	\lim_{t \rightarrow T} \frac{T-t}{F(u(y\sqrt{T-t},t))} = 1, \quad \text{where } F(u) = \int_{u}^{\infty} \frac{ds}{f(s)}.
 \end{equation*}
 We note that this result corresponds to the one which is proved by Liu in 1989 for the case of $f(u) = e^u$, which has the scale invariance property unlike our super-exponential case.
 To prove the main result, we see the equation as a perturbation of the equation with $f(u) = e^u$ through a transformation introduced by Fujishima and Ioku in 2018 and estimate the additional term which appears after the transformation.  
\end{abstract}

\section{Introduction}
We consider the following semilinear heat equation:
\begin{align}\label{SHE}
	\begin{cases}
	u_t =\Delta u + f(u)
	&\mathrm{in}\ \Omega \times (0,T), \\[1mm]
	u = 0 
	&\mathrm{on} \ \partial \Omega \times (0,T), \\[1mm]
	u(0,\cdot)=u_0
	&\mathrm{in}\ \Omega,
	\end{cases}
\end{align} 
where $f(u) = e^{|u|^{p-1}u}|u|^{q-1}u\ (p>1, q \in \{0\}\cup [1,\infty)),\ \Omega = B_R(0) \subset \R^n \ (n\geq 1,\ R>0),\ u_0 \in L^\infty (\Omega)$ and $T = T(u_0)$ denotes the maximal existence time of the solution $u$.
We further assume that $u_0$ is a nonnegative and radially nonincreasing function, which implies that the solution $u$ also shares the same properties (see \cite[Proposition 52.17*]{QSbook}). \par
Let us define a function $F:(0,\infty) \rightarrow \R$ by 
\begin{equation}\label{defF}
	F(u) = \int_{u}^{\infty} \frac{ds}{f(s)}.
\end{equation}
It follows that the function $F$ is positive and monotonically decreasing, and thus admits the inverse function $F^{-1}$.
It is well known (see for example \cite{K63} or \cite[Theorem 17.3]{QSbook}) that for suitable initial data $u_0$, there exists a solution $u$ to \eqref{SHE} which blows up in finite time, i.e. $T(u_0)<\infty$ and $u$ satisfies

\begin{equation*}
	\lim_{t\rightarrow T} \|u(t)\|_{L^{\infty}} = \infty.
\end{equation*}

When the solution $u$ blows up in finite time, then $T$ is called the blow-up time. We say that $a \in \Omega$ is a blow-up point of $u$ if there exists a sequence $\{(a_i, t_i)\}_{i\in \N}$ such that $a_i \rightarrow a,\ t_i \rightarrow T$, and $|u(a_i,t_i)| \rightarrow \infty$ as $i \rightarrow \infty$.
If the blow-up solution $u$ satisfies 
\begin{equation}\label{type1}
		u(x,t) \leq F^{-1}(c(T-t)), \quad\   x \in \Omega,\ t\in (\tilde{t},T)
\end{equation}
for some $c>0$ and $\tilde{t} \in (0,T)$, then the blow-up is said to be of type I. If \eqref{type1} does not hold, then the blow-up is said to be of type II.
This blow-up rate is derived from the ODE: 
\begin{equation}\label{ODE}		
	y_t = f(y),\quad\ y(0) = y_0 > 0.
\end{equation}
The solution to $\eqref{ODE}$ is written by $y(t) = F^{-1}(T-t)$ for $T = F(y_0)$. 
By comparing the blow-up solution $u$ which blows up at $t = T \in (0,\infty)$ and the solution $y$ to \eqref{ODE} with $y_0 = F^{-1}(T)$, we obtain
\begin{equation}\label{belowrate}
	\|u(t)\|_{L^\infty(\Omega)} \geq F^{-1}(T-t)\quad\  \text{for}\ t \in (0,T) .
\end{equation}
Hence the estimate \eqref{type1} implies that the blow-up rate of the solution $u$ is the same as that of the ODE solution.

Let us review some known results about the blow-up behavior of the solution to the semilinear heat equation.
The blow-up behavior of the solution $u$ to \eqref{SHE} has been studied for decades in the cases of $f(u) = |u|^{p-1}u$ and $f(u) = e^u$.
When $f(u) = |u|^{p-1}u$, Giga and Kohn proved in \cite{GK87} that if $\Omega = \R^n$ or bounded convex domain, $p \in (1,p_s)$ and $u_0 \geq 0$, then the blow-up is of type I, 
where $p_s$ denotes the Sobolev exponent, i.e. $p_s = \infty$ if $n \leq 2$ and $p_s = \frac{n+2}{n-2}$ if $n \geq 3$.
Later on, Giga, Matsui and Sasayama showed in \cite{GMS04} that in the Sobolev subcritical case, any sign-changing blow-up solution to \eqref{SHE} with $f(u) = |u|^{p-1}u$ and $\Omega = \R^n$ is of type I. 
The authors of \cite{GK87} also considered the asymptotic blow-up behavior in \cite{GK85,GK87,GK89}. They showed that for any blow-up point $a \in \Omega$, if $p \in (1,p_s)$ and $\Omega$ is star-shaped about $a$, then
\begin{align}\label{ssblowup}
	\lim_{t \rightarrow T} (T-t)^{\frac{1}{p-1}}u(y\sqrt{T-t}+a, t) = \left(\frac{1}{p-1}\right)^{\frac{1}{p-1}}
\end{align}
uniformly on compact sets of $\R^n$. This limit is called the constant self-similar blow-up profile.
Since $(p-1)^{-\frac{1}{p-1}}(T-t)^{-\frac{1}{p-1}}$ is equal to $F^{-1}(T-t)$ for $f(u) = |u|^{p-1}u$, the property \eqref{ssblowup} means that the blow-up solution $u$ behaves like the ODE solution in the space-time parabolas $|x-a| \sim \sqrt{T-t}$.
In addition, there are a lot of studies on the blow-up behavior with $f(u) = |u|^{p-1}u$ (see for instance \cite{HVpre,MM04,QSbook,S19,V92} and references therein).

In the case of exponential nonlinearity $f(u) = e^u$, Liu proved in \cite{L89} that under the conditions of $n \leq 2$ and type I blow-up, the nonnegative blow-up solution $u$ satisfies 
\begin{align}\label{ssblowup2}
	\lim_{t\rightarrow T}(T-t)e^{u(y\sqrt{T-t}+a, t)} = 1
\end{align}
uniformly on compact sets in $\R^n$ for any blow-up point $a \in \Omega$. 
This result corresponds to \eqref{ssblowup} in the case of $f(u) = |u|^{p-1}u$ because $-\log (T-t)$ is the ODE solution when $f(u) = e^u$.
As in the case of $f(u) = |u|^{p-1}u$, various results for the blow-up behavior with exponential nonlinearity have also been found (see for example \cite{FP08,S22} and references therein).\par
Here, these results were obtained by using the scale invariance property of the equation $u_t - \lap u = f(u)$, where $f(u) = |u|^{p-1}u$ or $f(u) = e^u$.
When $\Omega = \R^n$, if a function $u$ satisfies the equation with $f(u) = |u|^{p-1}u$ or $f(u) = e^u$, then the equation is invariant under the transformation
\begin{equation}\label{sipower}
	u_\lambda(x,t) = \lambda^{\frac{2}{p-1}}u(\lambda x, \lambda^2 t), \quad \lambda > 0 
\end{equation}
for the case of power nonlinearity $f(u) = |u|^{p-1}u$, and
\begin{equation}\label{siexp}
	u_\lambda(x,t) = u(\lambda x, \lambda^2 t) + 2\log \lambda, \quad \lambda >0
\end{equation}
for the case of exponential nonlinearity $f(u) = e^u$. 
Since we have not found a transformation for general nonlinearities such as $f(u) = e^{|u|^{p-1}u}\ (p>1)$ which yields the scale-invariance yet, the results corresponding to the cases of $f(u) = |u|^{p-1}u$ and $f(u) = e^u$ cannot be obtained immediately.\\
However, the blow-up analysis for \eqref{SHE} without the scale invariance property is developing in recent years.
Hamza and Zaag proved in \cite{HZ22} that any blow-up solution of the equation \eqref{SHE} in whole space for $f(u) = |u|^{p-1}u\log^a(2+u^2)$ with $p \in (1,p_s)$ and $a \in \R$ is of type I.
Moreover, Fujishima and Kan studied the blow-up rate of the blow-up solution to \eqref{SHE} with general class of $f$. 
They showed in \cite{FK25} that there does not exist a nonnegative radial type II blow-up solution if $f$ satisfies some assumptions such as
\begin{equation*}
	\lim_{u \rightarrow \infty} \frac{f(u)}{f_0(u)} = 1,
\end{equation*}
where $f_0 \in C^1([0,\infty))$ is a positive function which satisfies 
\begin{align}\label{genf'Flimit}
  f_0'(u)F_0(u) \rightarrow \beta \in [1,\infty)\quad \ \text{as} \ u \rightarrow \infty,
\end{align}
under certain restrictions on $\beta$ and the space dimension $n$, and $F_0$ denotes the function defined by replacing $f$ with $f_0$ in \eqref{defF}.
In addition, Chabi proved the property corresponding to the constant self-similar blow-up profile for nonnegative radial blow-up solution to \eqref{SHE} with $f(u) = u^pL(u)$ and $f(u) = e^uL(e^u)$,
where $L$ is a function with slow variation and satisfies some additional assumptions (see \cite{C25-1,C25-2} for the details). 
Note that these assumptions to the nonlinearity include the functions such as $f(u) = u^p\log^a(K+u)$ with $p\in (1,p_s),\ K>1,\ a\in \R$ and $f(u) = e^uu^q$ with $q \geq 1$ and the study \cite{C25-1} deal with the nonradial blow-up solution.
Moreover, by using the results in \cite{C25-1,C25-2}, the more precise blow-up profile is considered in \cite{C25-2,CS25}.
The proofs in \cite{C25-1,C25-2,CS25} rely on the specific structure of the nonlinearities, thus it cannot be adapted to our super-exponential case $f(u) = e^{|u|^{p-1}u}|u|^{q-1}u$.

The main purpose of this paper is to obtain the asymptotic behavior of the blow-up solution to \eqref{SHE} with super-exponential nonlinearities.
The following is the main theorem of this paper. 

\begin{thm}\label{qssblowup}
Let $p > 1,\ q \in \{0\} \cup [1,\infty),\  n \leq 2$ and $u$ be the solution to \eqref{SHE} with $T<\infty$. Assume that the blow-up solution $u$ is of type I.
Then
\begin{equation}\label{eq:qssblowup}
	\lim_{t \rightarrow T} \frac{T-t}{F(u(y\sqrt{T-t},t))} = 1
\end{equation}
uniformly on compact sets $|y| \leq C$.
\end{thm}

\begin{remark}
  If $q = 0$, Friedman--McLeod proved in \cite{FM85} that if $u_0 \in C^2(\bar{\Omega})$ satisfies $\lap u_0 + f(u_0) \geq 0$ in $\Omega$ 
	(this implies $u_t \geq 0$ in $\Omega \times (0,T)$, see for example \cite[Proposition 52.19]{QSbook}), then the blow-up is of type I for any $n\geq 1$ and $p \geq 1$.
	We note that the facts in this remark are also true for general smooth bounded domain $\Omega$ and nonnegative initial data $u_0 \in L^\infty(\Omega)$.
\end{remark}

\begin{remark}
   When $p=1$ and $q=0$ (i.e. $f(u) = e^u$), we have $F(u) = e^{-u}$, therefore Theorem \ref{qssblowup} corresponds to \eqref{ssblowup2} for the radially symmetric blow-up solution. 
   In addition, we note that this theorem covers the nonlinearities which differ from those considered in previous works \cite{C25-1,C25-2,GK85,GK87,GK89,L89}.
\end{remark}

Let us explain the idea to prove Theorem \ref{qssblowup}. 
The difficulty in proving Theorem \ref{qssblowup} lies in the lack of self-similarity.
The main idea of the proof is transforming the equation \eqref{SHE} to the equation which is close to the one for $f(u) = e^u$.
For general nonlinearities, Fujishima considered the following transformation in \cite{F14}, which generalizes \eqref{sipower} and \eqref{siexp}, and is called quasi-scaling.
For a function $u$ and $\lambda > 0$, let us define $u_\lambda$ by 
\begin{equation}\label{qsi}
	u_\lambda(x,t) = F^{-1}(\lambda^{-2}F(u(\lambda x, \lambda^2 t))).
\end{equation}
In the exponential case $f(u) = e^u$, we have $F(u) = e^{-u}$ and $F^{-1}(u) = -\log u$, therefore \eqref{qsi} is equal to \eqref{siexp}. Therefore we can consider that this $u_\lambda$ is a generalization of \eqref{siexp}.
We can similarly see that \eqref{qsi} is equal to \eqref{sipower} if $f(u) = |u|^{p-1}u$.
If $u$ satisfies the equation $u_t -\lap u = f(u)$, then this $u_\lambda$ satisfies the following equation:
\begin{align}\label{eq:u_lambda}
  \dt u_\lambda - \Delta u_\lambda = f(u_\lambda) + \frac{|\nabla u_\lambda|^2}{F(u_\lambda)f(u_\lambda)}\Bigl(f'(u)F(u)-f'(u_\lambda)F(u_\lambda)\Bigr).
\end{align}
If $f(u) = |u|^{p-1}u$ or $f(u) = e^u$, then $f'F$ becomes the constant function, hence we obtain the scale invariance.
In \cite{F14}, the author considered the location of the blow-up set of the equation with small diffusion and general nonlinearities.
After that, Fujishima and Ioku modified the transformation \eqref{qsi} in \cite{FI18} to simplify the remainder term of \eqref{eq:u_lambda}. 
In this paper, we use the following transformation introduced in \cite{FI18}:
\begin{equation}\label{qsi2}
	u_\lambda(x,t) = -\log F(u(\lambda x, \lambda^2 t)) + 2\log \lambda.
\end{equation}
Note that we use ``$-\log$'' instead of ``$F^{-1}$'' in \eqref{qsi} and ``$-\log$'' is equal to the ``$F^{-1}$ when $f(u) = e^u$''. 
If the function $u$ satisfies the equation $u_t - \lap u = f(u)$, then the function $u_\lambda$ satisfies the following equation:
\begin{align*}
  \dt u_\lambda - \Delta u_\lambda = e^{u_\lambda} + |\nabla u_\lambda|^2(f'(u)F(u) - 1).
\end{align*}
In our case $f(u) = e^{|u|^{p-1}u}|u|^{q-1}u$, we can see that 
\begin{equation}\label{f'Flimit}
 f'(u)F(u) \rightarrow 1 \quad \ \text{as}\ u \rightarrow \infty
\end{equation} 
(we prove \eqref{f'Flimit} later in Lemma \ref{fproperty1}), hence it can be conjectured that the blow-up behavior of the solution to \eqref{SHE} is close to the exponential case in some sense.
Therefore we can expect that the method in \cite{L89} also works in our super-exponential case $f(u) = e^{u^p}u^q$.
However, the remainder term including $f'F-1$ does not appear in \cite{L89} due to the scale-invariance, thus we have to control this term in some way.
More precisely, we have to show that the term $f'(u(x,t))F(u(x,t))-1$ converges to $0$ rapidly enough as $(x,t)$ goes to $(0,T)$ along the curve $|x| = (-\log (T-t))^\alpha$ for some $\alpha > 0$. 
We note that the authors of \cite{FI18} considered the existence of the time-local solution of the heat equation with general nonlinearities in uniformly local Lebesgue space.
The transformation \eqref{qsi2} is also applied to the study of the existence of the time-global solution for small data in \cite{FI22,FI23} and the solution to the elliptic equation in \cite{M18}. 
As for the blow-up problems, this transformation was applied in \cite{FK25} for the case of $\beta = 1$ in \eqref{genf'Flimit}.
However, the study of the blow-up behavior of the semilinear heat equation with general nonlinearities using the transformation \eqref{qsi} or \eqref{qsi2} is limited as far as we know.
In this paper, we show a new result for asymptotics of the blow-up solution to \eqref{SHE} in Theorem \ref{qssblowup}.  \\
This paper is organized as follows. 
In Section 2, we show some properties of nonlinearity $f(u) = e^{u^p}u^q$ for $u \geq 0$ and some results which will be needed for the proof of Theorem \ref{qssblowup}.
In Section 3, we prove a lemma to control the term including $f'F-1$ and then show Theorem \ref{qssblowup}.
Throughout this paper, $C>0$ denotes a constant which may vary from line to line and we write $B_r = B_r(0)$ for $r>0$.

\section{Preliminaries}
In this section, we prove some properties of the nonlinearity $f$ and the blow-up solution $u$.

\begin{lemma}\label{fproperty1}
	If $p > 1$ and $q \in \{0\}\cup [1,\infty)$, then  there exists $l = l(p,q)> 0$ such that $f'(u)F(u) \leq 1$ for any $u \geq l$ and $f'(u)F(u) \rightarrow 1$ as $u \rightarrow \infty$.
	Moreover, there exists a constant $C = C(p,q)>0$ such that for any $u \geq l$,
	\begin{equation}\label{estf'F}
		1 - f'(u)F(u) \leq \frac{C}{pu^p + q}.
	\end{equation}
\end{lemma}

\begin{proof}
 Since $f'(u) = (pu^{p-1} + qu^{-1})e^{u^p}u^q$ and $pu^{p-1} + qu^{-1}$ is nondecreasing for $ u \geq \left(\frac{q}{p(p-1)}\right)^{1/p}$, we have by direct computation that
	\begin{align*}
		f'(u)F(u) = (pu^{p-1} + qu^{-1})e^{u^p}u^q\int_{u}^{\infty} e^{-s^p}s^{-q}\, ds 
																															&\leq e^{u^p}u^q \int_{u}^{\infty} (ps^{p-1} + qs^{-1})e^{-s^p}s^{-q}\, ds\\
																															&= e^{u^p}u^q \int_{u}^{\infty} -\left(e^{-s^p}s^{-q}\right)'\, ds \\
																															&= 1
	\end{align*}
	for $ u \geq \left(\frac{q}{p(p-1)}\right)^{1/p}$. Next, by using l'H\^{o}pital's rule, we have
	\begin{equation}\label{f'Flimitcalc}
    \begin{split}
      		\lim_{u\rightarrow \infty} f'(u)F(u) &= \lim_{u\rightarrow \infty} \frac{F(u)}{\frac{1}{f'(u)}}\\
																				 &= \lim_{u\rightarrow \infty} \frac{f'(u)^2}{f(u)f''(u)}\\
																				 &= \lim_{u\rightarrow \infty} \frac{(pu^{p-1} + qu^{-1})^2(e^{u^p}u^q)^2}{\{p(p-1)u^{p-2} - qu^{-2} + (pu^{p-1} + qu^{-1})^2\}(e^{u^p}u^q)^2}\\
																				 &= 1.
    \end{split}
	\end{equation}
 Finally, we show the inequality \eqref{estf'F}. By using the integration by parts, we have
 \begin{equation}
	\begin{split}\label{Fcalc}
			F(u) = \int_{u}^{\infty} e^{-s^p}s^{-q}\, ds &= \int_{u}^{\infty}\frac{ps^{p-1} + qs^{-1}}{ps^{p-1} + qs^{-1}} e^{-s^p}s^{-q}\, ds  \\
																				 &= \left[-\frac{1}{(ps^{p-1} + qs^{-1})e^{s^p}s^q}\right]_{u}^{\infty} + \int_{u}^{\infty} \left(\frac{1}{ps^{p-1} + qs^{-1}}\right)'e^{-s^p}s^{-q}\, ds\\
																				 &= \frac{1}{f'(u)} + \int_{u}^{\infty} \left(\frac{p(1-p)s^p + q}{(ps^{p} + q)^2}\right)e^{-s^p}s^{-q}\, ds.\\
	\end{split}
 \end{equation}
 By multiplying \eqref{Fcalc} by $f'(u)$, it follows that 
 \begin{align*}
	1 - f'(u)F(u) &= (pu^{p-1} + qu^{-1})e^{u^p}u^q\int_{u}^{\infty} \left(\frac{p(p-1)s^p - q}{(ps^{p} + q)^2}\right)e^{-s^p}s^{-q}\, ds.
 \end{align*}
 It can be easily checked that there exists $l_0 > 0$ such that $ \frac{p(p-1)s^p - q}{(ps^{p} + q)^2} $ is monotonically decreasing for $s \geq l_0$.
 Therefore if $ u \geq l_0$, then we have
 \begin{align*}
	1 - f'(u)F(u) \leq \frac{p(p-1)u^p - q}{pu^p+q}\cdot \frac{(pu^{p-1} + qu^{-1})f(u)F(u)}{pu^p+q}.
 \end{align*}
 Since $(pu^{p-1} + qu^{-1})f(u)F(u) = f'(u)F(u) \leq 1$ for $u \geq \left(\frac{q}{p(p-1)}\right)^{1/p}$ and $ \frac{p(p-1)u^p - q}{pu^p+q}$ is bounded in $[l_0, \infty)$, 
 \eqref{estf'F} holds for $u \geq \max \left\{l_0, \left(\frac{q}{p(p-1)}\right)^{1/p}\right\} \eqcolon l$.
 The proof is completed.
\end{proof}

\begin{remark}\label{f'Fbound}
  We note that when $q = 0$, we can take $l = 0$ and the function $f'F$ is bounded in $(0,\infty)$ because $f'(u)F(u) \rightarrow 0$ as $u \rightarrow 0$.
  If $q > 1$, then since $f'(u) \rightarrow 0$ and $F(u) \rightarrow \infty$ as $u \rightarrow 0$, the calculation \eqref{f'Flimitcalc} also holds as $u \rightarrow 0$ and we have 
  \begin{equation*}
    \lim_{u\rightarrow 0}  f'(u)F(u) = \frac{q}{q-1},
  \end{equation*} 
  therefore the function $f'F$ is also bounded in $(0,\infty)$ for the case of $q > 1$.
  This limit diverges to infinity when $q = 1$, thus we cannot obtain the boundedness of $f'F$ near $u = 0$.
\end{remark}

The following lemma provides the asymptotics of $F^{-1}(u)$ as $u$ goes to zero.
\begin{lemma}\label{fproperty2}
	If $f(u) = e^{u^p}u^q\,(p > 1,\ q \in \{0\}\cup [1,\infty))$, then
	\begin{equation}\label{asympF-1}
		\lim_{u \rightarrow 0+} \frac{F^{-1}(u)}{(-\log u)^{\frac{1}{p}}} = 1.
	\end{equation}
\end{lemma}

\begin{proof}
	By Lemma \ref{fproperty1}, $\log (pu^{p-1+q}+qu^{q-1}) + u^p + \log F(u) \rightarrow 0$ as $u \rightarrow \infty$. Since $\frac{\log (pu^{p-1+q}+qu^{q-1})}{u^p} \rightarrow 0$ as $u \rightarrow \infty$, we have
 \begin{equation*}
	\lim_{u \rightarrow \infty} \frac{-\log F(u)}{u^p} = 1.
 \end{equation*}
 Therefore, for any $\ep \in (0,1)$, there exists $M>0$ such that if $u > M$, then
 \begin{equation*}
	(1-\ep)u^p < -\log F(u) < (1+\ep)u^p,
 \end{equation*}
which is equivalent to
\begin{equation}\label{Festimate}
	e^{-(1+\ep)u^p} < F(u) < e^{-(1-\ep)u^p}.
\end{equation}
In general, if two decreasing functions $f$ and $g$ satisfy $f \leq g$ on an interval $I$, then $f^{-1} \leq g^{-1}$ on $f(I) \cap g(I)$. Hence we obtain from \eqref{Festimate} that
\begin{equation*}
	\left(\frac{1}{1+\ep}\right)^{\frac{1}{p}}\left(-\log u\right)^{\frac{1}{p}} < F^{-1}(u) < \left(\frac{1}{1-\ep}\right)^{\frac{1}{p}}\left(-\log u\right)^{\frac{1}{p}}
\end{equation*}
for sufficiently small $u$. This implies \eqref{asympF-1}.
\end{proof}

\begin{remark}
	The proof here also holds for the case $p \in (0,1]$ since $f'(u)F(u) \rightarrow 1$ as $u \rightarrow \infty$ remains true by the calculation \eqref{f'Flimitcalc}.
\end{remark}
The following proposition provides a derivative estimate for the type I blow-up solution $u$. 
\begin{prop}\label{estderivative}
	Let $u$ be the type I blow-up solution to \eqref{SHE} with $T < \infty$. Then for any $0 < r < \tilde{R} \eqcolon \min\{R, \sqrt{T}\}$, we have
	\begin{equation}\label{eq:estderivative}
		\frac{|\nabla u|}{f(u)F(u)} \leq C(T-t)^{-\frac{1}{2}}, \  \frac{|\nabla^2u|}{f(u)F(u)}\leq C(T-t)^{-1} \ \text{in}\  B_r\times [T-r^2,T).
	\end{equation}
\end{prop} 

The proof of this proposition is essentially based on \cite[Proposition 1]{GK85}, but due to the lack of scale invariance, the proof here requires a slight modification.

\begin{proof}
	We may assume that $\frac{\tilde{R}}{2} < r < \tilde{R}$. Set $D \coloneq B_r \times [T-r^2, T-\frac{1}{2}(\tilde{R}-r)^2]$, then $\overline{D}$ is a compact set of $\Omega \times (0,T)$. 
	Since the solution $u$ is bounded on $\overline{D}$, the parabolic regularity theorem implies $u \in C^{2+\alpha,1+\frac{\alpha}{2}}(\overline{D})$ for some $\alpha \in (0,1)$. In particular, we obtain that for some constant $C>0$,
	\begin{equation*}
		|\nabla u| + |\nabla^2 u| \leq C \quad\ \text{for}\  (x,t) \in D.
	\end{equation*}
	When $(x,t) \in B_r \times (T-\frac{1}{2}(\tilde{R}-r)^2,T)$, fix such $(x,t)$ and let 
	\begin{equation*}
		 \lambda = \left(\frac{2}{\tilde{R}^2}(T-t)\right)^{\frac{1}{2}}.
	\end{equation*}
	Define a function $w$ by 
	\begin{equation*}
		w(z,\tau) \coloneq \frac{\lambda^2}{F(u(x+ \lambda z, (1-\lambda^2)T + \lambda^2 \tau))}.
	\end{equation*}
	Then the domain of $w(z,\tau)$ includes $B_{\tilde{R}} \times [T-\tilde{R}^2,T)$. Indeed, we can see that $2(T-t) < (\tilde{R}-r)^2$, 
	therefore we have
	\begin{equation*}
		|x + \lambda z| \leq |x|+\left(\frac{2}{\tilde{R}^2}(T-t)\right)^{\frac{1}{2}} \tilde{R} < r + (\tilde{R} - r) = \tilde{R} \quad\ \text{for}\ z \in B_{\tilde{R}}.
	\end{equation*} 
	Moreover, if $\tau \in [T-\tilde{R}^2,T)$, then $(1-\lambda^2)T + \lambda^2 \tau < T$ and 
	\begin{align*}
		(1-\lambda^2)T + \lambda^2 \tau &\geq \left(1- \frac{2}{\tilde{R}^2}(T-t)\right)T + \frac{2}{\tilde{R}^2}(T-t)(T-\tilde{R}^2)\\
																		&= T - 2(T-t)\\
																		&> T-(\tilde{R}-r)^2\\
																		&> T-\tilde{R}^2.
	\end{align*}
  By direct calculations, we have
  \begin{align}
    w_\tau = \lambda^2 w \frac{u_t}{f(u)F(u)},\notag\\
    \nabla_z w = \lambda w \frac{\nabla_xu}{f(u)F(u)} \label{wderiv}
  \end{align}
  and 
  \begin{equation*}
	 \nabla_z^2 w = \lambda^2 w \left(\frac{\partial_{x_ix_j}u}{f(u)F(u)} - \frac{(f'(u)F(u)-2)\partial_{x_i}u\,\partial_{x_j}u}{f^2(u)F^2(u)}\right)_{1\leq i,j \leq n}.
  \end{equation*}
	Therefore the function $w$ satisfies the equation
	\begin{equation}\label{weq}
		w_{\tau} = \lap w + w^2 + \frac{|\nabla w|^2}{w}\left(f'(u)F(u) - 2\right).
	\end{equation}
	Since the blow-up solution $u$ is of type I, it follows from \eqref{type1} and the decreasing property of $F$ that
	\begin{equation*}
		\begin{split}
			w(z,\tau) &= \frac{\lambda^2}{F(u(x+ \lambda z, (1-\lambda^2)T + \lambda^2 \tau))}\\
								&\leq \frac{\lambda^2}{c((T - (1-\lambda^2)T - \lambda^2 \tau))}\\
								&= \frac{1}{c(T - \tau)}.
		\end{split}
	\end{equation*}
	Set $K \coloneq \overline{B_{\frac{\tilde{R}}{2}}} \times [T - \frac{3\tilde{R}^2}{4}, T-\frac{\tilde{R}^2}{4}]$, then $K$ is a compact set of $B_{\tilde{R}} \times [T-\tilde{R}^2,T)$ and 
	\begin{equation}\label{west}
		0 \leq w(z,\tau) \leq \frac{4}{c\tilde{R}^2} \quad \ \text{in}\ K.
	\end{equation}
	In order to use the parabolic regularity theorem to \eqref{weq}, we need to show the boundedness of $|\nabla w|/w$ (and $f'(u(x+ \lambda z, (1-\lambda^2)T + \lambda^2 \tau))F(u(x+ \lambda z, (1-\lambda^2)T + \lambda^2 \tau)) - 2$ when $q = 1$) in $K$.
  We note that when $q \in \{0\} \cup (1,\infty)$, the boundedness of $f'F-2$ is already obtained in Remark \ref{f'Fbound}.\\
	At first, we check the boundedness of $f'(u(x+ \lambda z, (1-\lambda^2)T + \lambda^2 \tau))F(u(x+ \lambda z, (1-\lambda^2)T + \lambda^2 \tau)) - 2$ when $q = 1$.
	By the strong maximum principle, the solution $u$ is positive on compact sets of $B_{\tilde{R}} \times [T - \tilde{R}^2, T)$, thus we may assume that there exists a constant $\dl > 0$ such that $u(x+\lambda z, (1-\lambda^2)T + \lambda^2 \tau) \geq \dl$ for $(z,\tau) \in K$.
	By Lemma \ref{fproperty1} and Remark \ref{f'Fbound}, we get the boundedness of $f'(u)F(u) - 2$ in $K$ for any $q \in \{0\} \cup [1,\infty)$.\\
  Next, we see the boundedness of $|\nabla w|/w$.
	Since $f(s)F(s) > 0$ for any $s>0$ and the solution $u(x+ \lambda z, (1-\lambda^2)T + \lambda^2 \tau)$ is bounded and positive in $K$, we get
	\begin{align*}
		f(u)F(u) \geq m \quad\ \text{in}\ K
	\end{align*}
	for some $m > 0$. 
	In addition to this, \eqref{wderiv}, \eqref{west} and the boundedness of $|\nabla u|$ on compact sets of $B_{\tilde{R}} \times [T-\tilde{R}^2, T)$ imply that $|\nabla w|$ is also bounded in $K$.
	Moreover, since $u(x+\lambda z, (1-\lambda^2)T + \lambda^2 \tau) \geq \dl$ in $K$, we have
	\begin{align*}
			w(z,\tau) &= \frac{\lambda^2}{F(u(x+ \lambda z, (1-\lambda^2)T + \lambda^2 \tau))}\\
								&\geq \frac{\lambda^2}{F(\dl)} > 0,
	\end{align*}
	hence we obtain the boundedness of $|\nabla w|/w$ in $K$.\\
	Therefore we can also apply the parabolic regularity theorem to \eqref{weq} and derive
 \begin{equation}\label{wderivest}
	|\nabla w| + |\nabla^2 w| <C \quad\ \text{in}\ K 
 \end{equation}
 for some constant $C>0$.
 We note that the constant $C$ in \eqref{wderivest} can be taken independently of $(x,t)$ since the choice of the compact set $K$ and the bound \eqref{west} are not depend on $(x,t)$.
 Here, when $(z,\tau) = (0, T-\frac{\tilde{R}^2}{2}) \in K$, we have $x + \lambda z = x$ and 
 \begin{align*}
  (1-\lambda^2)T + \lambda^2 \tau &= (1-\lambda^2)T + \lambda^2 \left(T - \frac{\tilde{R}^2}{2}\right)\\
                                  &= T - (T-t)\\
                                  &= t.
 \end{align*}
 Hence by using the boundedness of $f'F-2$ and the calculations of $\nabla w$ and $\nabla^2 w$, we obtain \eqref{eq:estderivative}.
\end{proof}


Next, we consider a transformation which will be used in the proof of the main theorem. 
\begin{prop}\label{qsstrans}
	Assume that a postive function u satisfies $u_t - \lap u = f(u)$ in $\Omega \times (0,T)$. Let $a \in \Omega$ and $v$ be a function defined by 
	\begin{equation}\label{vdef}
		v(y,s) = \frac{T-t}{F(u(x,t))},\quad\ y = \frac{x-a}{\sqrt{T-t}},\quad\   s = -\log (T-t).
	\end{equation}
	Then $v$ satisfies the following equation in $\Omega(s) \times (-\log T, \infty)$:
	\begin{equation}\label{veq}
		v_s = \Delta v -\frac{y}{2} \cdot \nabla v - \frac{|\nabla v|^2}{v} + v^2 - v + \frac{|\nabla v|^2}{v}\Bigl(f'(u)F(u) - 1\Bigr), 
	\end{equation}
	where $\Omega(s) = \{y \in \R^n \mid ye^{-\frac{s}{2}} \in \Omega \}$.
\end{prop}

\begin{remark}
  In \cite{L89}, this transformation was used for the case of $f(u) = e^u$. 
  Actually, in the case of $f(u) = e^u$, the function $v$ in \eqref{vdef} is equal to the function in the limit \eqref{ssblowup2} since $F(u) = e^{-u}$ when $f(u) = e^u$ as we already mentioned in Section 1.
	It can be considered that the transformation \eqref{vdef} is derived from \eqref{qsi2} by letting $\lambda = (T-t)^{-\frac{1}{2}}$ and applying exponential function to both sides of \eqref{qsi2}.
\end{remark}
\begin{proof}
	 Let $v$ be a function defined by \eqref{vdef}. Then $ u(x,t) = F^{-1}\left(\frac{T-t}{v(y,s)}\right)$. It follows from the direct computation that
 \begin{align*}
	u_t &= f(u)\left[\frac{1}{v} + \frac{T-t}{v^2}\left(\frac{v_s}{T-t} + \frac{y}{2(T-t)}\cdot \nabla_y v\right) \right],\\
	\nabla_x u &= f(u)\frac{\nabla_y v}{v^2}\sqrt{T-t},\\
	\lap_x u &= f'(u)\frac{\nabla_x u \cdot \nabla_y v}{v^2}\sqrt{T-t} + f(u)\left(\frac{\lap_y v}{v^2} - 2\frac{|\nabla_y v|^2}{v^3}\right)\\
				 &= f'(u)f(u)F(u) \frac{|\nabla_y v|^2}{v^3} + f(u)\left(\frac{\lap_y v}{v^2} - 2\frac{|\nabla_y v|^2}{v^3}\right).
 \end{align*}
 Therefore, it holds that
 \begin{align*}
	0 &= u_t - \lap u - f(u)\\
    &= f(u)\left[\frac{1}{v} + \frac{1}{v^2}\left(v_s + \frac{y}{2}\cdot \nabla v\right) \right] 
		 - f'(u)f(u)F(u) \frac{|\nabla v|^2}{v^3} - f(u)\left(\frac{\lap v}{v^2} - 2\frac{|\nabla v|^2}{v^3}\right) - f(u).
 \end{align*}
 Multiplying this equation by $\frac{v^2}{f(u)}$, we obtain that
 \begin{equation*}
	 v_s - \lap v + \frac{y}{2}\cdot \nabla v + \frac{|\nabla v|^2}{v} - v^2 + v -\frac{|\nabla v|^2}{v}\Big(f'(u)F(u)-1\Big) = 0.
 \end{equation*}
 We complete the proof.
\end{proof}

If $u$ goes to infinity, then the equation \eqref{veq} formally becomes
\begin{equation}\label{veqlim}
	v_s = \Delta v -\frac{y}{2} \cdot \nabla v - \frac{|\nabla v|^2}{v} + v^2 - v
\end{equation}

by Lemma \ref{fproperty1}. The next lemma asserts that if $n\leq2$, the stationary problem of \eqref{veqlim} has only the constant solution $v \equiv 1$. 
The following lemma is adapted from \cite[Lemma 3.1, Theorem 3.2]{L89}. 

\begin{lemma}\label{stasol}
	If $v$ is a positive bounded stationary solution to \eqref{veqlim} in $\R^n$ with $|\nabla v|/v \leq C$ for some constant $C>0$, then
	\begin{equation*}
		\frac{1}{4}\int_{\R^n} \frac{|\nabla v|^2}{v^2}|y|^2 \rho \, dy + \frac{2-n}{2}\int_{\R^n} \frac{|\nabla v|^2}{v^2} \rho \, dy = 0,
	\end{equation*}
	where $\rho(y) = e^{-\frac{|y|^2}{4}}$. Moreover, if $n \leq 2$, then $v \equiv 1$.
\end{lemma}

We end this section by introducing a lemma related to the set of blow-up points.
\begin{lemma}
	Under the assumptions of Theorem \ref{qssblowup}, the blow-up point of the solution $u$ is only at the origin.
\end{lemma}

For the proof of this lemma, see \cite[Section 4.2]{CS25} or \cite[Theorem 2.3]{FM85}. 
We remark that the function $f(u) = e^{u^p}u^q$ for $p > 1$ and $q \in \{0\} \cup [1,\infty)$ satisfies the assumptions of \cite[Section 4.2]{CS25}.

\section{Proof of Theorem \ref{qssblowup}}
In this section, we prove the main theorem. As a preliminary step, we prove a lemma that will be needed for the computations concerning an energy functional. 
In what follows, we set $s_0 \coloneq -\log T$.
\begin{lemma}\label{intdifferential}
	Let $\alpha \in (0,\infty)$. If $g:B_{s^\alpha} \times (s_0,\infty) \rightarrow \R$ is a smooth function, then 
	\begin{equation*}
		\frac{d}{ds}\int_{B_{s^\alpha}} g(y,s) \, dy = \int_{B_{s^\alpha}} g_s(y,s)\ dy + \frac{\alpha}{s}\int_{\partial B_{s^\alpha}} g(y,s) (y\cdot \nu) \, dS
	\end{equation*}
	for $s>0$, where $\nu$ is the unit outer normal vector to $\partial B_{s^\alpha}$.
\end{lemma}

\begin{proof}
	Set $y = s^{\alpha}z$, then
  \begin{equation*}
		\int_{B_{s^\alpha}} g(y,s) \, dy = \int_{B_1} g(s^{\alpha}z,s)s^{\alpha n} \, dz.
	\end{equation*}
	Therefore we obtain that
	\begin{align*}
		\frac{d}{ds}\int_{B_{s^\alpha}} g(y,s) \, dy &= \int_{B_1} g_s(s^\alpha z,s)s^{\alpha n} \, dz \\
																								 & \quad + \int_{B_1} \left(\nabla g(s^\alpha z,s) \cdot z\right)\alpha s^{\alpha n + \alpha - 1} \, dz + \int_{B_1} g(s^\alpha z,s)\alpha ns^{\alpha n-1}  \, dz\\
																								 &= \int_{B_{s^\alpha}} g_s(y,s)\, dy + \frac{\alpha}{s}\int_{B_{s^\alpha}} \nabla g(y,s)\cdot y\, dy + \frac{\alpha n}{s}\int_{B_{s^\alpha}} g(y,s) \, dy\\
																								 &= \int_{B_{s^\alpha}} g_s(y,s)\, dy + \frac{\alpha}{s}\int_{\partial B_{s^\alpha}} g(y,s) (y\cdot \nu) \, dS.
	\end{align*}
	In the last equality, we applied the integration by parts to the second term. The proof is completed.
\end{proof}

The following lemma assures that the solution $u(x,t)$ diverges to infinity as $t \rightarrow T$ along the curve $|x| = (-\log (T-t))^\alpha \sqrt{T-t}$.
It is the most important lemma in this paper to complete the proof of Theorem \ref{qssblowup}.

 \begin{lemma}\label{esth}
	Under the assumptions of Theorem \ref{qssblowup}, let $u = u(x,t) = u(|x|,t)$ be the type I blow-up solution to \eqref{SHE} and $\alpha \in (0,\frac{1}{p})$. Define $h_\alpha(t)$ by 
	\begin{equation*}
		h_\alpha(t) \coloneq u((-\log (T-t))^\alpha\sqrt{T-t}, t), \quad\ t\in (\max\{T-1,0\},T).
	\end{equation*}
	Then there exists $t_1 \in (\max\{T-1,0\},T)$ such that 
	\begin{equation}\label{eq:esth}
		h_\alpha(t) \geq \frac{1}{2}F^{-1}(T-t)\quad\ \text{for} \ t \in (t_1,T).
	\end{equation}
 \end{lemma}

 \begin{proof}
	Assume that \eqref{eq:esth} fails, then there exists a sequence $\{t_k\}_{k \in \N}$ such that $t_k \rightarrow T$ as $k \rightarrow \infty$ and
	\begin{equation*}
		h_\alpha(t_k) < \frac{1}{2}F^{-1}(T-t_k)\quad\  \text{for}\ k \in \N.
	\end{equation*}
	Then \eqref{belowrate} and \eqref{asympF-1} imply that
	\begin{align*}
		u(0,t_k) - h_\alpha(t_k) > \frac{1}{2}F^{-1}(T-t_k) > \frac{1}{4}(-\log (T-t_k))^{\frac{1}{p}}
	\end{align*}
	for sufficiently large $k$. On the other hand, 
	by using \eqref{eq:estderivative}, we have
	\begin{align*}
		u(0,t) - h_\alpha(t) &= \int_{0}^{(-\log (T-t))^\alpha\sqrt{T-t}} -u_r(r,t)\,dr\\
												 &\leq \int_{0}^{(-\log (T-t))^\alpha\sqrt{T-t}} Cf(u(r,t))F(u(r,t))(T-t)^{-\frac{1}{2}}\, dr\\
												 &\leq C\sup_{r \in (0,(-\log (T-t))^\alpha\sqrt{T-t}]} f(u(r,t))F(u(r,t))(-\log (T-t))^\alpha
	\end{align*} 
	for $t \in (\max \{T-1,0\},T)$. Hence we obtain
	\begin{equation*}
		\frac{1}{4}(-\log (T-t_k))^\frac{1}{p} \leq C\sup_{r \in (0,(-\log (T-t_k))^\alpha\sqrt{T-t_k}]} f(u(r,t_k))F(u(r,t_k))(-\log (T-t_k))^\alpha.
	\end{equation*}
	Since $\alpha \in (0,\frac{1}{p})$, it follows that 
  \begin{equation*}
		\lim_{k \rightarrow \infty} \sup_{r \in (0,(-\log (T-t_k))^\alpha\sqrt{T-t_k}]} f(u(r,t_k))F(u(r,t_k)) = \infty.
	\end{equation*}
	However, this contradicts the fact that $f(u)F(u)$ is bounded near $u = 0$. 
  Indeed, when $q = 0$, we obtain $F(0) < \infty$ and $f(0) = 1$, therefore the boundedness near $u = 0$ holds.
	If $q \geq 1$, since $f(u) \rightarrow 0$ and $F(u) \rightarrow \infty$ as $u \rightarrow 0$, it follows from l'H\^{o}pital's rule that 
	 \begin{align*}
		\lim_{u \rightarrow 0+} f(u)F(u) &= \lim_{u \rightarrow 0+} \frac{F(u)}{\frac{1}{f(u)}}\\
																		&= \lim_{u \rightarrow 0+} \frac{f(u)}{f'(u)}\\
																		&= \lim_{u \rightarrow 0+} \frac{e^{u^p}u^q}{(pu^{p-1} + qu^{-1})e^{u^p}u^q}\\
																    &= 0,
	\end{align*}
	therefore we obtain the boundedness of $f(u)F(u)$ near $u = 0$ and the conclusion follows.
 \end{proof}

Let $\alpha \in (0,\infty)$ and $\rho = \rho (y) = e^{-\frac{|y|^2}{4}}$. 
For a smooth function $v:\Omega (s) \times (s_0,\infty) \rightarrow \R$, where $\Omega(s) = \{y \in \R^n \mid ye^{-\frac{s}{2}} \in \Omega \}$, define an energy functional $E[v]$ by
\begin{equation*}
	E[v](s) \coloneq \frac{1}{2} \int_{B_{s^\alpha}} \frac{|\nabla v|^2}{v^2}\rho \, dy - \int_{B_{s^\alpha}} (v - \log v)\rho \, dy
\end{equation*}
for $s>0$.
Note that the integration domain is usually taken to be $\Omega(s)$, but for the sake of computational convenience we consider the energy functional on $B_{s^\alpha} (\subset \Omega(s)$ for sufficiently large $s$) instead.\par
The following lemma provides a kind of energy estimate for the solution to \eqref{SHE}, which plays an important role in the proof of the main theorem.
The calculations of the proof in this lemma are based on \cite[Lemma 4.4]{L89}. However, due to the lack of self-similarity, we need to estimate the term including $f'F-1$, which did not appear in \cite{L89}.

\begin{lemma}\label{estenergy}
	Under the assumptions of Theorem \ref{qssblowup}, if $\alpha \in (0,\frac{1}{p})$, then $v(y,s)$ defined by \eqref{vdef} with $a = 0$ satisfies
	\begin{equation*}
	  \frac{1}{2}\int_{B_{s^\alpha}} \frac{v_s^2}{v^2}\rho\, dy \leq -\frac{d}{ds}E[v](s) + H(s) \quad\ \text{for}\  s > s_0',
	\end{equation*}
	where $s_0' \coloneq \max \{s_0, 0\}$ and $H(s)$ is a function satisfying $\int_{s_0'}^{\infty} H(s)\, ds < \infty$.
\end{lemma}

\begin{proof}
Let $u$ be the blow-up solution to $\eqref{SHE}$ and define $v(y,s)$ by \eqref{vdef} with $a=0$. Multiplying \eqref{veq} by $\frac{v_s}{v^2}\rho$ and integrating over $B_{s^\alpha}$, we get
\begin{align}
		\int_{B_{s^\alpha}} \frac{v_s^2}{v^2}\rho\ dy &= \int_{B_{s^\alpha}} \frac{v_s}{v^2}\nabla \cdot (\rho\nabla v)\,dy - \int_{B_{s^\alpha}} \frac{|\nabla v|^2}{v^3}v_s\rho\,dy \notag \\
																				   &\quad + \int_{B_{s^\alpha}} \left(1 - \frac{1}{v} \right)v_s\rho\,dy + \int_{B_{s^\alpha}} \frac{|\nabla v|^2}{v^2}\Bigl(f'(u)F(u) - 1\Bigr)\frac{v_s}{v}\rho\,dy \notag \\
																					 &= -\int_{B_{s^\alpha}} \left(\nabla \left(\frac{v_s}{v^2}\right)\cdot \nabla v\right) \rho \,dy + \int_{\partial B_{s^\alpha}} \frac{v_s}{v^2} \frac{\partial v}{\partial \nu} \rho \,dS 
																					   - \int_{B_{s^\alpha}} \frac{|\nabla v|^2}{v^3}v_s\rho\,dy \notag \\
																				   &\quad + \int_{B_{s^\alpha}} \left(1 - \frac{1}{v} \right)v_s\rho\,dy + \int_{B_{s^\alpha}} \frac{|\nabla v|^2}{v^2}\Bigl(f'(u)F(u) - 1\Bigr)\frac{v_s}{v}\rho\,dy \notag \\
                                           &= -\frac{1}{2}\int_{B_{s^\alpha}} \frac{d}{ds}\left(\frac{|\nabla v|^2}{v^2}\right)\rho\,dy + \int_{B_{s^\alpha}} \frac{d}{ds}(v - \log v)\rho\,dy \notag \\
																				   &\quad + \int_{\partial B_{s^\alpha}} \frac{v_s}{v^2}\frac{\partial v}{\partial \nu} \rho \, dS + \int_{B_{s^\alpha}} \frac{|\nabla v|^2}{v^2}\Bigl(f'(u)F(u) - 1\Bigr)\frac{v_s}{v}\rho\,dy\notag \\
                                           &= -\frac{d}{ds}E[v](s) + G(s) + \int_{B_{s^\alpha}} \frac{|\nabla v|^2}{v^2}\Bigl(f'(u)F(u) - 1\Bigr)\frac{v_s}{v}\rho\,dy, \label{energy} 
 \end{align}
where
\begin{equation*}
	G(s) = \int_{\partial B_{s^\alpha}} \frac{v_s}{v^2}\frac{\partial u}{\partial \nu} \rho \,dS - \frac{\alpha}{s} \int_{\partial B_{s^\alpha}} \left(\frac{1}{2}\frac{|\nabla v|^2}{v^2} - v + \log v\right)(y \cdot \nu) \rho\, dS.
\end{equation*}
In the last equality, we used Lemma \ref{intdifferential}. Applying Young's inequality to the last term, we have
\begin{equation*}
    \frac{1}{2}\int_{B_{s^\alpha}} \frac{v_s^2}{v^2}\rho\, dy 
                                                               \leq -\frac{d}{ds}E[v](s) + H(s),
\end{equation*} 
where
\begin{equation}\label{defH}
	H(s) \coloneq G(s) + \frac{1}{2}\int_{B_{s^\alpha}}\left(\frac{|\nabla v|}{v}\right)^4\Bigl(f'(u)F(u) - 1\Bigr)^2 \rho\, dy.
\end{equation}
Let us see the integrability of the function $H$. First, we estimate the function $G(s)$. 
We note that by Lemma \ref{esth}, we have 
\begin{equation}\label{udiverge}
  u(s^\alpha e^{-\frac{s}{2}}, T-e^{-s}) \rightarrow \infty\quad\  \text{as}\  s \rightarrow \infty.
\end{equation}
It follows from the radially decreasing property of the solution $u$, Lemma \ref{fproperty1} and \eqref{udiverge} that
\begin{equation}\label{f'F-1bounded}
	|f'(u(ye^{-\frac{s}{2}}, T-e^{-s}))F(u(ye^{-\frac{s}{2}}, T-e^{-s})) - 1| \leq C
\end{equation}
for sufficiently large $s>0$ and $y \in B_{s^\alpha}$.
Computations in the proof of Proposition \ref{qsstrans}, \eqref{eq:estderivative} and \eqref{f'F-1bounded} imply that there exists a constant $C > 0$ such that
\begin{equation}\label{vdiffest}
	\frac{|\nabla v|}{v} \leq C,\quad\ \frac{|\lap v|}{v} \leq C\quad \ \text{in}\ B_{s^\alpha}
\end{equation}
 for sufficiently large $s>0$. Moreover, since the solution $u$ is of type I, the function $v$ is also bounded. Indeed, by applying the function $F$ to the inequality \eqref{type1}, we have $F(u(x,t)) \geq c(T-t)$ for $t \in (\tilde{t},T)$.
 Therefore from the equation \eqref{veq}, Remark \ref{f'Fbound}, \eqref{f'F-1bounded} and \eqref{vdiffest}, we have
 \begin{equation}\label{vs}
	\begin{split}
		\frac{|v_s|}{v} &\leq \frac{|\lap v|}{v} + \frac{|y|}{2}\frac{|\nabla v|}{v} + \frac{|\nabla v|^2}{v^2} + v + 1 + \frac{|\nabla v|^2}{v^2}|f'(u)F(u)-1|\\
	        &\leq C(1+|y|)
	\end{split}
 \end{equation}
 in $B_{s^\alpha}$ for sufficiently large $s > 0$. 
 Define $w(y,s) = \log v(y,s),$ then by applying the function $-\log F$, which is nondecreasing, to \eqref{belowrate}, we have $w(0,s) \geq 0$ for $s>s_0$. Since $\nabla w = \frac{|\nabla v|}{v}$ and \eqref{vdiffest} holds, we get
 \begin{align}\label{wbelowest}
  |w(y,s) - w(0,s)| \leq C |y|
 \end{align}
 for sufficiently large $s$. Moreover, It follows from \eqref{type1} that the function $w$ is bounded above in the same way to get the boundedness of $v$.
 Therefore we have
	\begin{equation}\label{vbelow}
		-C|y| \leq w(y,s) \leq C \quad \ \text{in}\ B_{s^\alpha}
  \end{equation} 
  for sufficiently large $s$.
 By combining the boundedness of $v$, \eqref{vdiffest}, \eqref{vs} and \eqref{vbelow}, we obtain that
 \begin{align*}
	\left| \int_{\partial B_{s^\alpha}} \frac{v_s}{v^2}\frac{\partial v}{\partial \nu} \rho \,dS \right| 
			&\leq \int_{\partial B_{s^\alpha}} \left|\frac{v_s}{v}\right|\frac{|\nabla v|}{v}\rho\,dS\\
			&\leq C(s^\alpha + 1)s^{\alpha(n-1)}e^{-\frac{s^{2\alpha}}{4}},\\
	\left| \int_{\partial B_{s^\alpha}} \log v(y,s) (y \cdot \nu)\rho \, dS\right| &\leq Cs^{\alpha(n+1)}e^{-\frac{s^{2\alpha}}{4}},\\
	\left| \int_{\partial B_{s^\alpha}} \frac{|\nabla v|^2}{v^2} (y \cdot \nu)\rho \, dS\right| &\leq Cs^{\alpha n}e^{-\frac{s^{2\alpha}}{4}}
 \end{align*}
 and 
 \begin{equation*}
		\left| \int_{\partial B_{s^\alpha}} v (y \cdot \nu)\rho \, dS\right| \leq  Cs^{\alpha n}e^{-\frac{s^{2\alpha}}{4}}.
 \end{equation*}
 Hence we get
 \begin{equation}\label{estG(s)}
	\int_{s_0'}^{\infty} G(s)\, ds < \infty.
 \end{equation}

 Let us estimate the remaining term including $f'F-1$ by using Lemma \ref{esth}.
 It follows from \eqref{vdiffest} that
 \begin{align*}
	\int_{B_{s^\alpha}}\left(\frac{|\nabla v|}{v}\right)^4\Bigl(f'(u)F(u) - 1\Bigr)^2 \rho\, dy 
										&\leq CI,
 \end{align*}
 where
 \begin{equation*}
	I = \int_{B_{s^\alpha}}\left(f'(u)F(u) - 1\right)^2 \rho\, dy.
 \end{equation*}
 Lemma \ref{fproperty1}, Lemma \ref{fproperty2}, Lemma \ref{esth} and the radially decreasing property of the solution $u$ imply that there exist $s_1 > 0$ such that if $s \geq s_1$, then
 \begin{align*}
	I &\leq \int_{B_{s^{\alpha}}} \frac{C}{(pu^{p}(ye^{-\frac{s}{2}},T-e^{-s})+q)^2} \rho\,dy\\
			&\leq \frac{C}{(ph_{\alpha}^{p}(T-e^{-s})+q)^2}\int_{B_{s^{\alpha}}} \rho\,dy\\
			&\leq \frac{C}{\left(\frac{p}{2}(F^{-1}(e^{-s}))^{p} + q\right)^2}\\
			&\leq \frac{C}{\left(\frac{p}{4}s + q\right)^2}.
 \end{align*}
 Hence we have
 \begin{equation}\label{estf'F-1}
	\int_{s_0'}^{\infty} \int_{B_{s^\alpha}}\left(\frac{|\nabla v|}{v}\right)^4\Bigl(f'(u)F(u) - 1\Bigr)^2 \rho\, dy\,ds < \infty
 \end{equation}
 and it follows from \eqref{defH}, \eqref{estG(s)} and \eqref{estf'F-1}, the function $H$ is integrable on $(s_0',\infty)$.
\end{proof}

 To complete the proof of Theorem \ref{qssblowup}, we need the following lemma.
 \begin{lemma}\label{vinftypositive}
	Let $\{s_j\}_{j\in\N}$ be an increasing sequence such that $s_j \rightarrow \infty$ as $j \rightarrow \infty$ and assume that 
	$v_j(y,s) \coloneq v(y,s+s_j)$ converges to a function $v_\infty(y,s)$ uniformly on compact sets in $\R^{n+1}$, then either $v_\infty \equiv 0$ or $v_\infty > 0$ in $\R^{n+1}$.
 \end{lemma}

 This lemma is taken from \cite[Lemma 4.3]{L89}, but for the reader's convenience, we provide a proof.

 \begin{proof}
	Let $y_1,y_2 \in \Omega(s)$. Replacing $|w(y,s) - w(0,s)|$ in \eqref{wbelowest} with $|w(y_1,s) - w(y_2,s)|$, we have
	\begin{equation*}
		\left|\log \frac{v(y_1,s)}{v(y_2,s)}\right| = |w(y_1,s) - w(y_2,s)| \leq C|y_1 - y_2|.
	\end{equation*}
	Therefore we obtain
	\begin{equation*}
		v(y_1,s) \leq v(y_2,s)e^{C|y_1-y_2|} \quad\ \text{for}\  s \in (s_0, \infty).
	\end{equation*}
  Then by substituting $s+s_j$ for $s$ and letting $j \rightarrow \infty$, we have
	\begin{equation*}
		v_\infty(y_1,s) \leq v_\infty(y_2,s)e^{C|y_1-y_2|}\quad \text{for}\ y_1,y_2 \in \R^n,\ s \in \R.
	\end{equation*}
	Then this inequality implies that for any $s$, if there exists a point $y \in \R^n$ such that $v_\infty(y,s) =0$, then $v_\infty(\cdot,s) \equiv 0$.
	Similarly, By using \eqref{vs} instead of \eqref{vdiffest}, it follows that for any $s_1,s_2 \in (s_0,\infty)$,
	\begin{align*}
		\left|\log \frac{v(y,s_1)}{v(y,s_2)}\right| = |w(y,s_1) - w(y,s_2)| \leq C(1+|y|)|s_1 - s_2| \quad \text{for}\  y \in \Omega(s_1) \cap \Omega(s_2).
	\end{align*}
	Thus we have
	\begin{equation*}
		v(y,s_1) \leq v(y,s_2)e^{C(1+|y|)|s_1-s_2|} \quad\ \text{for}\ y \in \Omega(s_1) \cap \Omega(s_2),
	\end{equation*}
	and the conclusion follows from the same argument as above.
 \end{proof}



 \begin{proof}[Proof of Theorem \ref{qssblowup}]
	Let $\alpha = \frac{1}{2p}$ and $\{s_j\}_{j\in \N}$ be an increasing sequence such that $s_j \rightarrow \infty $ as $j \rightarrow \infty$ and define $v_j$ as in Lemma \ref{vinftypositive}.  
	Since $v_j$ and $\nabla v_j/v_j$ are bounded in $B_k \times (-k, k)$ and the boundedness is independent of $k \in \N$ by \eqref{vdiffest} and the same agument to obtain \eqref{vs}, 
	the parabolic regularity theorem, the Arzel\`{a}--Ascoli theorem and the diagonal argument imply that there exists a function $v_\infty$ such that
	$v_j(y,s) \coloneq v(y,s+s_j) \rightarrow v_\infty$ in $C^{2,1}(\R^{n+1})$ as $j \rightarrow \infty$ by taking a subsequence, still denoted by $\{s_j\}_{j \in \N}$. 
	By taking a subsequence once more if needed, we may assume that $s_{j+1} - s_j \rightarrow \infty$ as $j \rightarrow \infty$.
	Let us prove that the function $v_\infty$ is independent of the choice of the sequence $\{s_j\}_{j\in \N}$ and then $v_\infty \equiv 1$.\par
	Let $m \in (0,\infty)$, then by integrating the inequality \eqref{energy} on $[s_j+m, s_{j+1}+m]$, we have
	\begin{equation*}
		\begin{split}
			\int_{m}^{s_{j+1}-s_j+m} \int_{B_{(s+s_j)^\alpha}} \frac{(v_j)_s^2}{v_j^2} \rho\, dy\,ds 
										\leq E[v_j](m) - E[v_{j+1}](m) + \int_{s_j+m}^{s_{j+1}+m} H(s)\, ds. 
		\end{split}
	\end{equation*}
	Let us show that $E[v_j] \rightarrow E[v_\infty]$. 
	Thanks to the estimate \eqref{vbelow}, we obtain by using Lebesgue's dominated convergence theorem that
	\begin{equation*}
		\int_{B_{s^\alpha}} \log v_j \rho \, dy \rightarrow \int_{B_{s^\alpha}} \log v_\infty \rho\, dy.
	\end{equation*}
	The convergence of remaining terms easily follows from Lebesgue's dominated convergence theorem by using the boundedness of $v$ and \eqref{vdiffest}, so we omit the details.
	Therefore, by using Lemma \ref{estenergy} and Fatou's lemma, we obtain that
	\begin{equation*}
		\int_{m}^{M} \int_{\R^n} \frac{(v_\infty)_s^2}{v_\infty^2} \rho\, dy\,ds = 0
	\end{equation*}
	for any $M>m$. Hence $(v_\infty)_s \equiv 0$ in $\R^{n}\times (0,\infty)$. By Lemma \ref{vinftypositive}, we have $v_\infty \equiv 0$ or $v_\infty > 0 $ in $ \R^{n+1}$.
	If $v_\infty \equiv 0$, then for any $\ep > 0$, $v_j(y,s) < \ep$ for $y \in B_1$ and sufficiently large $j\in \N$ (we can take $j$ independently of $y$ since $v_j$ converges to $v_\infty$ uniformly on compact sets in $\R^n$). 
	This implies 
	\begin{equation*}
		u(y\sqrt{T-t_j}, t_j) \leq F^{-1}\left(\frac{1}{\ep}(T-t_j)\right),
	\end{equation*}
	where $t_j = T - e^{-(s+s_j)}$. If we take $y = 0$ and $\ep < 1$, then this contradicts \eqref{belowrate} since $F^{-1}$ is decreasing. Therefore we have $v_\infty > 0$.
	For fixed $y\in \R^n$, it holds that $|y| \leq (-\log (T-t))^\alpha$ for any $t$ sufficiently close to $T$ depending on $y$, 
	therefore we obtain from Lemma \ref{esth} and radially nonincreasing property of the solution $u$ that $u(y\sqrt{T-t},t) \rightarrow \infty$ in $B_{(-\log (T-t))^\alpha}$ as $t \rightarrow T$ and thus 
	\begin{equation*}
			f(u(y\sqrt{T-t},t))F(u(y\sqrt{T-t},t)) - 1 \rightarrow 0 \quad \text{in}\  B_{(-\log (T-t))^\alpha}\  \text{as}\ t \rightarrow T 
	\end{equation*}
	by Lemma \ref{fproperty1}.
  Since $v$ satisfies the equation \eqref{veq}, the limit $v_\infty$ is a positive bounded stationary solution to \eqref{veqlim}, then $v_\infty \equiv 1 $ in $\R^n \times (0,\infty)$ by Lemma \ref{stasol}.
	Here, we have just proved that for any sequence $\{s_j\}_{j \in \N}$, there exsits a subsequence of $\{s_j\}_{j \in \N}$ such that $v_j(y,s) = v(y,s+s_j) \rightarrow 1$ as $j \rightarrow \infty$ in $\R^n \times (0,\infty)$. 
	Hence it follows that $v(y,s) \rightarrow 1$ as $s \rightarrow \infty$ uniformly on compact sets of $\R^n$, which means \eqref{eq:qssblowup} holds. 
	We complete the proof of Theorem \ref{qssblowup}.
 \end{proof}

\vspace{5mm}
\noindent
{\bf Acknowledgments.}
The author would like to thank his supervisor Kentaro Fujie for his constructive advice and encouragement during the preparation of this paper.


\begin{thebibliography}{99}
%

\bibitem{C25-1}
 \sc L.D. Chabi, 
  \it Asymptotic blow-up behavior for the semilinear heat equation with non scale invariant nonlinearity.
	\rm  J. Differ. Equ. {\bf 435} (2025), 113303.


\bibitem{C25-2}
 \sc L.D. Chabi, 
  \it Refined Blow-up Behavior for Reaction-Diffusion Equations with Non Scale Invariant Exponential Nonlinearities.
	\rm  J. Dyn. Diff. Equat. (2025), to appear


\bibitem{CS25}
 \sc L.D. Chabi, Ph. Souplet,
  \it Refined behavior and structural universality of the blow-up profile for the semilinear heat equation with non scale invariant nonlinearity.
	\rm  Math. Ann. {\bf 391} (2025), 4509--4554.


\bibitem{FP08}
 \sc M. Fila, A. Pulkkinen, 
  \it Nonconstant selfsimilar blow-up profile for the exponential reaction-diffusion equation.
	\rm Tohoku Math. J.  {\bf 60} (2008), 303--328.


\bibitem{FM85}
 \sc A. Friedman, B. McLeod, 
  \it Blow-up of positive solutions of semilinear heat equations.
	\rm Indiana Univ. Math. J.  {\bf 34} (1985), 425--447.


\bibitem{F14}
 \sc Y. Fujishima,  
  \it Blow-up set for a superlinear heat equation and pointedness of the initial data.
	\rm Discrete Contin. Dyn. Syst. {\bf 34} (2014),  4617--4645.


\bibitem{FI18}
 \sc Y. Fujishima, N. Ioku, 
  \it Existence and nonexistence of solutions for the heat equation with a superlinear source term.
	\rm J. Math. Pures Appl. {\bf 118} (2018), 128--158.


\bibitem{FI22}
 \sc Y. Fujishima, N. Ioku, 
  \it Global in time solvability for a semilinear heat equation without the self-similar structure.
	\rm Partial Differ. Equ. Appl. {\bf 3} (2022), 23.


\bibitem{FI23}
 \sc Y. Fujishima, N. Ioku, 
  \it Quasi self-similarity and its application to the global in time solvability of a superlinear heat equation.
	\rm Nonlinear Anal. {\bf 236} (2023), 113321.


\bibitem{FK25}
 \sc Y. Fujishima, T. Kan, 
  \it Uniform boundedness and blow-up rate of solutions in non-scale-invariant superlinear heat equations.
	\rm J. Elliptic Parabol. Equ. (2025), to appear.


\bibitem{GK85}
 \sc Y. Giga, R.V. Kohn, 
  \it Asymptotically self-similar blow-up of semilinear heat equations.
	\rm Commun. Pure Appl. Math. {\bf 38} (1985), 297--319.  


\bibitem{GK87}
 \sc Y. Giga, R.V. Kohn, 
  \it Characterizing Blow-up using Similarity Variables.
	\rm Indiana Univ. Math. J. {\bf 36} (1987), 1--40.
	

\bibitem{GK89}
 \sc Y. Giga, R.V. Kohn, 
  \it Nondegeneracy of blowup for semilinear heat equations.
	\rm Commun. Pure Appl. Math. {\bf 45} (1989), 845--884.


\bibitem{GMS04}
 \sc Y. Giga, S. Matsui, S. Sasayama,
  \it Blow up rate for semilinear heat equations with subcritical nonlinearity.
	\rm Indiana Univ. Math. J. {\bf 53} (2004), 483--514.


\bibitem{HZ22}
 \sc M.A. Hamza, H. Zaag, 
  \it The blow-up rate for a non-scaling invariant semilinear heat equation.
	\rm Arch. Ration. Mech. Anal. {\bf 244} (2022), 87--125.


\bibitem{HVpre}
 \sc M.A. Herrero, J.J.L. Vel\'{a}zquez,
 \it A blow up result for semilinear heat equations in the supercritical case.
 \rm (preprint).


\bibitem{K63}
 \sc S. Kaplan, 
  \it On the growth of solutions of quasi-linear parabolic equations.
	\rm Commun. Pure Appl. Math. {\bf 16} (1963), 305--330.


\bibitem{L89}
 \sc W. Liu, 
  \it The blow-up rate of solutions of semilinear heat equations.
	\rm J. Differ. Equ. {\bf 77} (1989), 104--122.  


\bibitem{MM04}
 \sc H. Matano, F. Merle, 
  \it On nonexistence of type II blowup for a supercritical nonlinear heat equation.
	\rm Commun. Pure Appl. Math. {\bf 57} (2004), 1494--1541.  


\bibitem{M18}
 \sc Y. Miyamoto, 
 \it A limit equation and bifurcation diagrams for semilinear elliptic equations with general supercritical growth.
 \rm J. Differ. Equ. {\bf 264} (2018), 2684--2707.


\bibitem{QSbook}
 \sc P. Quittner, Ph. Souplet, 
  \it Superlinear Parabolic Problems. Blow-up, Global Existence and Steady States. Second Edition.
	\rm Birkh\"{a}user Advanced Texts, (2019).


\bibitem{S19}
 \sc Ph. Souplet, 
  \it  A simplified approach to the refined blowup behavior for the nonlinear heat equation.
	\rm  SIAM J. Math. Anal. {\bf 51} (2019), 991--1013.


\bibitem{S22}
 \sc Ph. Souplet, 
  \it  On refined blowup estimates for the exponential reaction-diffusion equation.
	\rm  Partial Differ. Equ. Appl. {\bf 3} (2022), 16.

\bibitem{V92}
 \sc J.J.L. Vel\'{a}zquez,
 \it Higher-dimensional blow up for semilinear parabolic equations.
 \rm Comm. Partial Differ. Equ. {\bf 17} (1992), 1567--1596. 


\end{thebibliography}
\end{document}